\documentclass[11pt]{article}
\usepackage{amsmath, amssymb, amscd, amsthm, amsfonts}
\usepackage{graphicx}
\usepackage{hyperref}
\usepackage{enumerate}
\usepackage{dsfont}	
\usepackage{indentfirst}	
	
\usepackage{latexsym}

\title{ \textbf{Existence, asymptotic behaviors, and high-dimensional uniqueness of topological solutions to the skew-symmetric  Chern-Simons system on lattice graphs}}

\newtheorem{theorem}{Theorem}[section]

\newtheorem{lemma}[theorem]{Lemma}
\newtheorem{corollary}[theorem]{Corollary}

\newtheorem{prop}[theorem]{Proposition}

\newcommand{\tm}{\begin{theorem}}
\newcommand{\tmd}{\end{theorem}}
\newcommand{\co}{\begin{corollary}}
\newcommand{\cod}{\end{corollary}}
\newcommand{\prp}{\begin{prop}}
\newcommand{\prpd}{\end{prop}}

\usepackage{mathrsfs}	
\usepackage{titletoc}	
\numberwithin{equation}{section}

\newcommand{\ra}{\rightarrow}

\renewcommand{\l}{\lambda}	
	
\newcommand{\be}{\begin{equation}}	
\renewcommand{\ra}{\rightarrow}	
\newcommand{\ee}{\end{equation}}	
\newcommand{\bea}{\begin{eqnarray}}	
\newcommand{\eea}{\end{eqnarray}}	
\newcommand{\bna}{\begin{eqnarray*}}	
\newcommand{\ena}{\end{eqnarray*}}

\begin{document}
\author{Honggang Liu \footnote{Email address:hgliu23@m.fudan.edu.cn }}
\date{}	
\maketitle	

\begin{center}	
School of Mathematical Sciences, Fudan University, Shanghai, China	
\end{center}	

\date{}
\maketitle
\begin{abstract}
In this paper, we consider the topological 
solutions to the skew-symmetric  Chern-Simons system on lattice graphs:
$$\left\{\begin{aligned}		
	\Delta u &=\lambda\mathrm{e}^{\upsilon}(\mathrm{e}^{u}-1)+4\pi\sum\limits_{j=1}^{k_1}m_j\delta_{p_j},\\	
	\Delta \upsilon&=\lambda\mathrm{e}^{u}(\mathrm{e}^{\upsilon}-1)+4\pi\sum\limits_{j=1}^{k_2}n_j\delta_{q_j},	
\end{aligned}		
\right.		
$$	
here, $\lambda\in\mathbb{R}_+$, $k_1$ and $k_2$ are two positive integers,  $m_j\in\mathbb{N}\, (j=1,2,\cdot\cdot\cdot,k_1)$, $n_j\in\mathbb{N}\,(j=1,2,\cdot\cdot\cdot,k_2)$, and $\delta_{p}$ denotes the Dirac mass at vertex $p$. Write
$$g=4\pi\sum_{j=1}^{k_1}m_j\delta_{p_j},\ h=4\pi\sum_{j=1}^{k_2}n_j\delta_{q_j},\ B=4\pi\sum_{j=1}^{k_1}m_j+4\pi\sum_{j=1}^{k_2}n_j.$$
For any fixed $g,h$, we prove the existence of the topological solutions to the systems, then obtain the asymptotic behaviors of topological solutions as $\l \ra 0_+$ and $\l \ra +\infty$, and finally prove the uniqueness of the topological solutions when the dimension of lattice graph $\mathbb{Z}^n$ is large enough or $\lambda$ is large enough.
\end{abstract}

	

\section{Introduction}\label{section-introduction}

The Chern-Simons-Higgs model, introduced by Hong, Kim, Pac \cite{introchern-simons1} and Jackiw, Weinberg \cite{introchern-simons2}, has always attracted the attention of many mathematicians in the fields of geometry and physics (see for examples\cite{introchern-simons3, introchern-simons4}). The model has also been an important topic in physics to study multiply distributed electrically and magnetically charged vortices(\cite{Back1},\cite{Back2},\cite{Back3},
\cite{Back4},\cite{Back5}).
In $\mathbb{R}^2$, many scholars have considered the skew-symmetric  Chern-Simons system:	
\be\label{uve}	
\left\{\begin{aligned}	
\Delta u &=\lambda\mathrm{e}^{\upsilon}(\mathrm{e}^{u}-1)+4\pi\sum\limits_{j=1}^{k_1}m_j\delta_{p_j},\\	
	\Delta \upsilon&=\lambda\mathrm{e}^{u}(\mathrm{e}^{\upsilon}-1)+4\pi\sum\limits_{j=1}^{k_2}n_j\delta_{q_j},		
\end{aligned}	
\right.	
\ee	
 where $\lambda\in\mathbb{R}_+$, $k_1$ and $k_2$ are two positive integers,  $m_j\in\mathbb{N}\, (j=1,2,\cdot\cdot\cdot,k_1)$, $n_j\in\mathbb{N}\,(j=1,2,\cdot\cdot\cdot,k_2)$, and $\delta_{p}$ denotes the Dirac mass at the vertex $p$. A solution of $(\ref{uve})$ is called topological if $u(x)\rightarrow0$ as $|x|\rightarrow+\infty$, and nontopological if $u(x)\rightarrow-\infty$ as $|x|\rightarrow+\infty$. This system arises from the self-dual equations for the relativistic Abelian Chern–Simons model with two Higgs fields and two gauge fields (\cite{Chern-Simonstheory1,Chern-Simonstheory2}). After a reduction process, the governing system of the self-dual equations for this model reduces to the system $(\ref{uve})$. The process can be referred to \cite{Chern-Simonstheory} in detail. The existence of topological solutions in $\mathbb{R}^2$ was also established in \cite{Chern-Simonstheory}, and see \cite{sup-subsolutions,sup-subsolutions-6,sup-subsolutions-7,sup-subsolutions-8,sup-subsolutions-9} for other related results.

 Recently, people have paid attention to the system $(\ref{uve})$ on finite graphs. In the paper \cite{JFA}, Huang, Wang and Yang studied the existence of maximal condensates, and also established the existence of multiple solutions, including a local minimizer for the 
 transformed energy functional and a mountain-pass-type solution. In the paper \cite{Topologicaldegree}, by calculating the topological degree, and using the relation between the degree and the critical group of a related functional, Li, Sun and Yang obtained multiple solutions of the system $(\ref{uve})$. Chao, Hou and Sun studied the existence of solutions to a generalized self-dual Chern-Simons system on finite graphs; see e.g. \cite{generalizedsystem}, and see \cite{generalizedsystem2,topologicaldegreget3solutions} for other related results.

 In paper \cite{Hua2023TheEO}, Hua, Huang and Wang considered topological solutions to the self-dual Chern-Simons vortex equation on $\mathbb{Z}^n$($n\geqslant2$):
\begin{equation*}
\left\{
\begin{aligned}
& \Delta u=\lambda e^u(e^u-1)+4\pi\sum_{j=1}^Mn_j\delta_{p_j}\ \ \text{on} \ \mathbb{Z}^n,\\
& \lim_{d(x)\rightarrow+\infty}u(x)=0.\\
\end{aligned}
\right.
\end{equation*}
They also constructed a topological solution, which is maximal among all possible solutions.
In this paper, we will consider the topological solutions of (\ref{uve}) on $\mathbb{Z}^n$.

Let $\mathbb{Z}^n= (V, E)$ be the integer lattice graph, where			
\[			
V = \left\{x: x = (x_1, \ldots, x_n)  \text{ where }  x_i \in \mathbb{Z} \text{ for each } 1\leq i\leq n\right\}\subseteq \mathbb{R}^n,			
\]			
\[			
E = \left\{ xy : x, y \in V \text{ such that } d(x,y) = 1 \right\},			
\] and \[			
d(x, y) = \sum_{i=1}^n \left| x_i - y_i \right|.			
\]			
We write $d(x)=d(x,0)$. For any function $u:\mathbb{Z}^n\rightarrow \mathbb{R}$, the $l^p$-norm of $u$ is defined as
\begin{equation*}
\|u\|_{l^p(\mathbb{Z}^n)}=\left\{
\begin{aligned}
& \left(\sum_{x\in\mathbb{Z}^n}|u(x)|^p\right)^{\frac{1}{p}},\ 1\leqslant p<\infty,\\
& \sup_{x\in\mathbb{Z}^n}|u(x)|, \ p=\infty.\\
\end{aligned}
\right.
\end{equation*}
Let $e_i$ be the vector whose $i$-th component is 1 and the others are 0.
The Laplacian is defined as
$$\Delta u(x) = \sum_{y\sim x}(u(y)-u(x))=\sum_{i=1}^n(h(x+e_i)+h(x-e_i)-2h(x)), \forall x\in\mathbb{Z}^{n},$$
where $y\sim x$ means $xy\in E.$

In the following we mainly consider topological solutions to the skew-symmetric  Chern-Simons system on $\mathbb{Z}^n$($n\geqslant2$):
\begin{equation}
\label{uvf}
\left\{
\begin{aligned}
&\Delta u =\lambda\mathrm{e}^{\upsilon}(\mathrm{e}^{u}-1)+4\pi\sum\limits_{j=1}^{k_1}m_j\delta_{p_j} \ \ \ \ \text{on} \ \mathbb{Z}^n,\\	
	&\Delta \upsilon=\lambda\mathrm{e}^{u}(\mathrm{e}^{\upsilon}-1)+4\pi\sum\limits_{j=1}^{k_2}n_j\delta_{q_j} \ \  \ \ \text{on} \ \mathbb{Z}^n,
\\
 &\lim_{d(x)\rightarrow+\infty}u(x)=0,\lim_{d(x)\rightarrow+\infty}v(x)=0,\\
\end{aligned}
\right.
\end{equation}
where $\lambda\in\mathbb{R}_+$, $k_1$ and $k_2$ are two positive integers, $m_j\in\mathbb{N}\, (j=1,2,\cdot\cdot\cdot,k_1)$, $n_j\in\mathbb{N}\,(j=1,2,\cdot\cdot\cdot,k_2)$ and $\delta_{p}$ denotes the Dirac mass at vertex $p$. Recall that
\begin{equation}
\label{constantB}
\begin{aligned}
g=4\pi\sum_{j=1}^{k_1}m_j\delta_{p_j},\ h=4\pi\sum_{j=1}^{k_2}n_j\delta_{q_j},\ B=4\pi\sum_{j=1}^{k_1}m_j+4\pi\sum_{j=1}^{k_2}n_j.
\end{aligned}
\end{equation}

In the papar, our first theorem extends the results in \cite{Hua2023TheEO} from a single equation to system (\ref{uvf}). 
\tm\label{thm:main1}
System $(\ref{uvf})$ has a topological solution $(u^\ast,v^\ast)\in l^p(\mathbb{Z}^n)\times l^p(\mathbb{Z}^n)$ for $1\leqslant p\leqslant\infty$ and $n\geqslant 2$, which is maximal among all possible solutions and satisfies $$u^*\leqslant0,\ \ v^\ast\leqslant0\ \ \ \ \mathrm{on}\  \mathbb{Z}^n.$$ Furthermore, for any topological solution $(u,v)$ of $(\ref{uvf})$ and $\epsilon\in(0,1)$, $as \ d(x)\rightarrow+\infty,$
$$u= O(e^{-m(1-\epsilon) d(x)}),v= O(e^{-m(1-\epsilon) d(x)}),$$
where $m=\ln(1+\frac{\lambda}{2n})$.
\tmd
In the second theorem, we obtain the asymptotic behaviors of topological solutions as $\l \ra 0_+$ and $\l \ra +\infty.$ When $\l \ra +\infty$, this conclusion is consistent with the results in $\mathbb{R}^2$, see Lemma 3.1 in \cite{asymptoticR}. In the paper \cite{generalizedsystem2}, Hou and Kong also investigated the asymptotic behaviors of solutions to Chern-Simons system on finite graphs.
\begin{theorem}\label{thm1.2}
For fixed $g$ and $h$, let $(u_\lambda,v_\lambda)$ be a topological solution of $(\ref{uve})$ on $\mathbb{Z}^n$, then 
\begin{enumerate}[(a)]
\item If $\lambda>2B(2n+e^{4B})$, then 
$$u_\lambda(x)+v_\lambda(x)\geqslant \ln(1-\frac{2B}{\lambda}), \ \ \forall x\in \mathbb{Z}^n,$$
where $B$ is given by $(\ref{constantB})$. Consequently, $u_\lambda,\ v_\lambda \rightarrow 0\ \ as \ \lambda\rightarrow+\infty.$
\item 
If $n=2,$ then
\begin{equation*}
\lim_{\lambda\rightarrow 0_+}(u_\lambda(x),v_\lambda(x))=\left\{
\begin{aligned}
 (-\infty,-\infty) \ \ \ \     &\text{if} \ \ g\not\equiv0,h\not\equiv0,\\
 (-\infty,0) \ \ \ \ \ \quad    &\text{if} \ \ g\not\equiv0,h\equiv0,\\
 (0,-\infty) \ \ \ \ \ \quad    &\text{if} \ \ g\equiv0,h\not\equiv0,\\
 (0,0) \ \ \ \ \quad \ \  \quad  &\text{if} \ \ g\equiv0,h\equiv0.\\
\end{aligned}
\right.
\end{equation*}

If $n\geqslant3,$ suppose $G_n$ is the Green’s function of $\Delta$ operator on $\mathbb{Z}^n$ $($see $(\ref{e:Green function})$ for definition$)$, then for $\forall x\in\mathbb{Z}^n,$ 
$$\lim_{\lambda\rightarrow 0_+}(u_\lambda(x),u_\lambda(x))=(4\pi\sum\limits_{j=1}^{k_1}m_jG_n(x-p_j),4\pi\sum\limits_{j=1}^{k_2}n_jG_n(x-q_j)).$$
\end{enumerate}
\tmd
So far, there have been few results on the uniqueness of the topological solution of (\ref{uve}). It is established that the topological solution for (\ref{uve}) in $\mathbb{R}^2$ is unique when the parameter $\lambda$ is either sufficiently large or sufficiently small(see $\cite{asymptoticR}$). This paper extends this result by proving the uniqueness of the topological solution for the lattice graph $\mathbb{Z}^n$ when the dimension $n$ is sufficiently large or when $\lambda$ is sufficiently large. Specifically, we establish the following theorem:
\tm\label{thm1.3}
There exist two constants $N(g,h)$ and $\lambda(g,h)$ such that if $n\geqslant N(g,h)$
or $\lambda\geqslant \lambda(g,h)$, then the topological solution of $(\ref{uvf})$ is unique.
\tmd

\section{Preliminaries}\label{section-Preliminaries}
We denote by $C(\mathbb{Z}^n)=\{u:\mathbb{Z}^n\rightarrow\mathbb{R}\}$ the set of functions on $\mathbb{Z}^n$, and by $supp(u)=\{x\in\mathbb{Z}^n:u(x)\neq0\}$ the support of $u$, and by $C_0(\mathbb{Z}^n)$ the set of functions with finite support.
For a finite subset $\Omega\subset\mathbb{Z}^n$, we define the boundary of $\Omega$ as
$$\delta\Omega:=\{y\in \mathbb{Z}^n\setminus \Omega:\exists x\in\Omega\ \text{such that}\ y\sim x\},$$
and write $\overline{\Omega}=\Omega\cup \delta\Omega$. For $u\in C(\Omega)$, the null extension to $\mathbb{Z}^n$ of $u$ is defined as
\begin{equation*}
\widetilde{u}(x)=\left\{
\begin{aligned}
u(x) \ \ \ \  \ &\mbox{in} \ \Omega,\\
0 \ \ \ \ \ \ \ \ \ \ &\text{in} \ \Omega^c.\\
\end{aligned}
\right.
\end{equation*}
We define the difference operator as
$$\nabla_{xy}u=u(y)-u(x),\ u\in C(\mathbb{Z}^n),\ \forall x,y\in\mathbb{Z}^n.$$
For $f,g\in C(\overline{\Omega})$, we introduce a bilinear form
$$D_{\Omega}(f,g):=\frac{1}{2}\sum_{\substack{x,y\in \Omega\\ x\sim y}}\nabla_{xy}f\nabla_{xy}g+\sum_{\substack{x\in \Omega,y\in \delta\Omega\\ x\sim y}}\nabla_{xy}f\nabla_{xy}g,$$
and we write $D_\Omega(f)=D_\Omega(f,f)$ for the Dirichlet energy of $f$ on $\Omega$. For $f\in C(\overline{\Omega})$, the directional derivative operator $\frac{\partial f}{\partial \vec{n}}$ at $x\in \delta\Omega$ is defined as
$$\frac{\partial f}{\partial \vec{n}}(x):=\sum_{\substack{y\in \Omega\\ x\sim y}}(f(x)-f(y)).$$
The followings are Green's identities on graphs, see e.g. \cite{bookanalysisgraphs}.
\begin{lemma}\label{lm1}
Let $f,g\in C(\mathbb{Z}^n)$ and $\Omega$ be a finite subset of $\mathbb{Z}^n$.
\begin{enumerate}[(a)]
\item If $f\in C_0(\mathbb{Z}^n)$, then
$$\frac{1}{2}\sum_{\substack{x,y\in \mathbb{Z}^n\\ x\sim y}}\nabla_{xy}f\nabla_{xy}g=-\sum_{x\in\mathbb{Z}^n}f(x)\Delta g(x).$$
\item $$D_\Omega(f,g)=-\sum_{x\in\Omega}f(x)\Delta g(x)+\sum_{x\in\delta\Omega}f(x)\frac{\partial g}{\partial \vec{n}}(x).$$
\end{enumerate}
\end{lemma}
 The following maximum principle plays a crucial role in this paper.
\begin{lemma}\label{lm2}
Let $\Omega$ be a connected finite subset of $\mathbb{Z}^n$. For any nonnegative $f\in C(\overline{\Omega})$, suppose that a function $v\in C(\overline{\Omega})$ satisfies
\begin{equation*}
\left\{
\begin{aligned}
 (\Delta-f)v\geqslant 0\ \ \ \ \  &\text{in}\ \Omega,\\
 v\leqslant0\ \ \ \ \ \quad \ \quad \quad &\text{on}\ \delta\Omega .\\
\end{aligned}
\right.
\end{equation*}
Then $v\leqslant0$ on $\overline{\Omega}$.
\end{lemma}

\begin{proof}

We prove the result by contradiction. Suppose that there exists $x_0\in \Omega$ such that $v(x_0)=\sup_{y\in\overline{\Omega}}v(y)=c>0$. By the equation,
$$\Delta v(x_0)\geqslant f(x_0)v(x_0)\geqslant0.$$
This implies that for any $x\sim x_0,\ x\in\overline{\Omega}$, we have $v(x)=v(x_0)=c.$ By the connectivity of $\overline{\Omega}$, we can see that for any $x\in\overline{\Omega},$ $v(x)=v(x_0)=c>0,$
which yields a contradiction.
\end{proof}

By the above lemma, we have the following corollaries.

\co\label{co1}
For any nonnegative $f\in C(\mathbb{Z}^n)$, suppose that a function $v\in C(\mathbb{Z}^n)$ satisfies
\begin{equation*}
\left\{
\begin{aligned}
& (\Delta-f)v\geqslant 0\ \ \text{on}\ \mathbb{Z}^n,\\
& \lim_{d(x)\rightarrow+\infty}v(x)\leqslant0 .\\
\end{aligned}
\right.
\end{equation*}
Then $v\leqslant0$ on $\mathbb{Z}^n$.
\cod
\

\co\label{Extre}
Suppose that a function $v\in C(\mathbb{Z}^n)$ satisfies
\begin{equation*}
\left\{
\begin{aligned}
& \Delta v= 0\ \ \ \  \text{on}\ \mathbb{Z}^n,\\
& \lim_{d(x)\rightarrow+\infty}v(x)=0 .\\
\end{aligned}
\right.
\end{equation*}
Then $v=0$ on $\mathbb{Z}^n$.
\cod
\
The isoperimetric inequality is well-known on $\mathbb{Z}^n$, see e.g. \cite{isopermetric}, we denote by $|K|$ the cardinality of the set $K$.
\begin{lemma}\label{lm:iso}
There exists a constant $C_n$ depending on the dimension $n$, such that, for any finite graph $\Omega\subset\mathbb{Z}^n$, 
$$|\delta\Omega|\geqslant C_n|\Omega|^{\frac{n-1}{n}}.$$
\end{lemma}	
Also, we have the following representation formula for Green's function $G_{n}.$ The existence and uniqueness of the Green's function were rigorously established in Proposition 4.5.1 of \cite{bookgreenfunction}, while an explicit formulation is detailed in Section 1 of \cite{Greenfunctionend}.
\begin{lemma}\label{lm:Green function}
Let $n\geq3$, then there exists a unique $C(\mathbb{Z}^n)$-solution $G_n$ of
\begin{equation}
\label{Dire}
\left\{
\begin{aligned}
& \Delta u=\delta_0\  \ \text{on}\ \mathbb{Z}^n,\\
& \lim_{d(x)\rightarrow+\infty}u(x)=0 .\\
\end{aligned}
\right.
\end{equation}
Moreover, $G_{n}(x)$ has the following form:
\begin{equation}				
\label{e:Green function}				
     G_n(x) =				
     \frac{-1}{(2\pi)^n} \int_{[-\pi,\pi]^n} \frac{e^{iz\cdot x}}{2n -2 \sum_{j=1}^n \cos z_j} \mathrm{d} z
\end{equation}				
where $x \in \mathbb{Z}^n$, $z = (z_1,\ldots,z_n)\in \mathbb{R}^n$.
 We call $G_n$ the Green’s function of $\Delta$ operator on $\mathbb{Z}^n$.	
\end{lemma}

\section{Proof of Theorem 1.1}\label{section-Proof of theorem 1.1}
The proof follows the methods in \cite{Hua2023TheEO}.
To prove Theorem~\ref{thm:main1}, we first consider an iterative sequence on a finite subset of $\mathbb{Z}^n$.
Let $\Omega_0$ be a finite subset of $\mathbb{Z}^n$, satisfying $\Omega_0\supset (\{p_j\}_{j=1}^{k_1} \cup \{q_j\}_{j=1}^{k_2}) $, and $\Omega$ be an arbitrary connected finite subset such that $\Omega_0\subset\Omega\subset\mathbb{Z}^n$. We choose a constant $L>2\lambda>0$. Let $u_0=v_0=0$ and consider the following iterative equations,
\begin{equation}
\label{e:uk}
\left\{
\begin{aligned}
& (\Delta-L) u_k=\lambda e^{v_{k-1}}(e^{u_{k-1}}-1)+g-Lu_{k-1}\ \ \  \text{in} \ \Omega,\\
& (\Delta-L) v_k=\lambda e^{u_{k-1}}(e^{v_{k-1}}-1)+h-Lv_{k-1}\ \ \ \text{in} \ \Omega,\\
& u_k =v_k=0\ \ \  \text{on} \ \delta\Omega.\\
\end{aligned}
\right.
\end{equation}

\begin{lemma}\label{lm4}
Let sequences $\{u_k\}$ and $\{v_k\}$ be given in $(\ref{e:uk})$. Then for each $k$, $u_k$ and $v_k$ are uniquely defined and
$$0=u_0\geqslant u_1\geqslant u_2\geqslant\ldots,$$
$$0=v_0\geqslant v_1\geqslant v_2\geqslant\ldots.$$
\end{lemma}

\begin{proof}
We use mathematical induction to prove this lemma. First we have
\begin{equation*}
\left\{
\begin{aligned}
& (\Delta-L) u_1=g\ \ \text{in} \ \Omega,\\
& u_1 =0\ \ \quad\quad\quad\ \text{on} \ \delta\Omega.\\
\end{aligned}
\right.
\end{equation*}
One easily sees the existence and uniqueness of the solution $u_1$ on $\Omega$. Using Lemma~\ref{lm2}, we obtain that $u_1\leqslant0$, similarly, $v_1\leqslant0$

Suppose that $0=u_0\geqslant u_1\geqslant u_2\geqslant\ldots\geqslant u_{i}$ and $0=v_0\geqslant v_1\geqslant v_2\geqslant\ldots\geqslant v_{i}$. Since
$$\lambda e^{v_i}(e^{u_i}-1)+g-Lu_i\in l^2(\Omega),\lambda e^{u_i}(e^{v_i}-1)+g-Lv_i\in l^2(\Omega),$$
we have the existence and uniqueness of the solution $(u_{i+1},v_{i+1})$. From $(\ref{e:uk})$, we get
\begin{equation*}
\begin{aligned}
(\Delta-L)(u_{i+1}-u_i)&=\lambda[e^{v_i}(e^{u_i}-1)-e^{v_{i-1}}(e^{u_{i-1}}-1)]-L(u_i-u_{i-1})\\
&\geqslant \lambda [e^{v_{i-1}}(e^{u_i}-1)-e^{v_{i-1}}(e^{u_{i-1}}-1)]-L(u_i-u_{i-1})\\
&= \lambda e^{v_{i-1}}(e^{u_i}-e^{u_{i-1}})-L(u_i-u_{i-1})\\
&=\lambda e^{v_{i-1}+\xi_{i-1}}(u_i-u_{i-1})-L(u_i-u_{i-1})\\
&=(\lambda e^{v_{i-1}+\xi_{i-1}}-L)(u_i-u_{i-1})\\
&\geqslant 0,
\end{aligned}
\end{equation*}
where $\xi_{i-1}$ is a function satisfying $u_i\leqslant\xi_{i-1}\leqslant u_{i-1}$. From this and Lemma~\ref{lm2}, we conclude immediately that $u_{i+1}\leqslant u_i.$ Similarly, we have  $v_{i+1}\leqslant v_i$. So, by mathematical induction, we have proved this lemma.
\end{proof}

The following lemma proves that the monotone sequences $\{u_k\}$ and $\{v_k\}$ are convergent.
\begin{lemma}\label{lm3.2}
Let $\{u_k\}$ and $\{v_k\}$ be the sequences defined by $(\ref{e:uk})$. Then, there holds
$$u_k\rightarrow u_\Omega, \  
 v_k\rightarrow v_\Omega \ \text{on}\ \overline{\Omega},$$
where $u_\Omega,v_\Omega$ satisfy
\begin{equation}
\label{e:3.2}
\left\{
\begin{aligned}
& \Delta u_\Omega=\lambda e^{v_\Omega}(e^{u_\Omega}-1)+g \ \ \text{in} \ \Omega,\\
& \Delta v_\Omega=\lambda e^{u_\Omega}(e^{v_\Omega}-1)+h \ \ \text{in} \ \Omega,\\
& u_\Omega =v_\Omega=0\ \ \text{on} \ \delta\Omega.
\end{aligned}
\right.
\end{equation}
\end{lemma}

\begin{proof}
Since $\{u_k\}$ and $\{v_k\}$ are monotone sequences, the pointwise limit $(u_\Omega,v_\Omega)$ of $(u_k,v_k)$ exists. It suffices to show that $u_\Omega$ and $v_\Omega$ are bounded. We first consider the set
$$B(\Omega)=\{x\in \Omega: \exists y\in \delta\Omega\ \text{such that}\ y\sim x\}.$$
Summing over $\Omega$ in $(\ref{e:uk})$, and by Lemma~\ref{lm1} we obtain
\begin{equation*}
\begin{aligned}
&\sum_{x\in \delta\Omega}\frac{\partial u_k}{\partial \vec{n}}(x)+\lambda\sum_{x\in\Omega}e^{v_k}(1-e^{u_k})\\
=&\sum_{x\in\Omega}g(x)+L\sum_{x\in\Omega}(u_k(x)-u_{k-1}(x))\leqslant4\pi\sum_{j=1}^{d_1}m_j<B.\\
\end{aligned}
\end{equation*}
This yields that
$$\sum_{x\in B(\Omega)}|u_k(x)|< B.$$
In particular, for $\forall x\in B(\Omega)$, the sequence $\{u_k(x)\}$ is uniformly bounded.

If $x_1\sim x_0$, $x_1\in\Omega$ and $x_0\in B(\Omega)$, 
we deduce that
$$\Delta u_k(x_0)=\sum_{y\sim x_0}(u_k(y)-u_k(x_0))\leqslant u_k(x_1)-2nu_k(x_0)\leqslant u_k(x_1)+2nB$$
and $u_k(x_1)<0$.
The equation (\ref{e:uk}) at $x_0$ shows that
\begin{equation*}
\begin{aligned}
|\Delta u_k(x_0)|&\leqslant L|u_k(x_0)-u_{k-1}(x_0)|+\lambda|e^{v_{k-1}}(e^{u_{k-1}}-1)|+|g(x_0)|\\
&\leqslant L|u_k(x_0)|+\lambda+B\leqslant (L+1)B+\lambda.
\end{aligned}
\end{equation*}
 Hence, the sequence $\{u_k(x_1)\}$ is uniformly bounded.

Since $\Omega$ is finite and connected, by repeating the above process a finite number of times, it can be concluded that $\{u_k(x)\}$ is uniformly bounded on $\Omega$.
 Similarly, $\{v_k(x)\}$ is also uniformly bounded on $\Omega$. In conclusion, this lemma is proved.
\end{proof}

Let $\Omega_i$, $1\leqslant i<\infty$, be finite and connected subsets, satisfying
$$\Omega_0\subset \Omega_1\subset\ldots\subset \Omega_k\subset\ldots,\ \ \bigcup_{i=1}^\infty\Omega_i=\mathbb{Z}^n.$$
We denote $u^i =u_{\Omega_i}$ for simplicity. To prove Theorem~\ref{thm:main1}, we also need the following lemma:

\begin{lemma}\label{lm3.3}
Let $\Omega(\supset\Omega_0)$ be a finite subset of $\mathbb{Z}^n$, and $\{u_k\}$ and $\{v_k\}$ be the sequence defined by $(\ref{e:uk})$. Suppose $V\in C(\overline{\Omega})$ satisfying
\begin{equation*}
\left\{
\begin{aligned}
& \Delta W\geqslant\lambda e^V(e^W-1)+g\ \ \text{in} \ \Omega,\\
& \Delta V\geqslant\lambda e^W(e^V-1)+h\ \ \text{in} \ \Omega,\\
& W(x) \leqslant0\ \ \text{on} \ \delta\Omega,\\
& V(x) \leqslant0\ \ \text{on} \ \delta\Omega,\\
\end{aligned}
\right.
\end{equation*}
then
$$0=u_0\geqslant u_1\geqslant\ldots\geqslant u_k\geqslant\ldots\geqslant u_\Omega\geqslant W,$$
$$0=v_0\geqslant v_1\geqslant\ldots\geqslant v_k\geqslant\ldots\geqslant v_\Omega\geqslant V.$$
\end{lemma}

\begin{proof}
First, one has
$$\Delta W\geqslant\lambda e^V(e^W-1)+g\geqslant\lambda e^V(e^W-1).$$
We claim that $\sup_{x\in\Omega}W(x)\leqslant0$. If not, choose $W(x_0)=\sup_{x\in\Omega}W(x)>0$ for some $x_0\in \Omega$. Then
$$0\geqslant\Delta W(x_0)\geqslant \lambda e^{V(x_0)}(e^{W(x_0)}-1)>0,$$
which yields a contradiction and proves the claim. Similarly, we have $\sup_{x\in\Omega}V(x)\leqslant0$.

Suppose that $W\leqslant u_k$ and $V\leqslant v_k$, then
\begin{equation*}
\begin{aligned}
(\Delta-L)(u_{k+1}-W)&\leqslant \lambda e^{v_k}(e^{u_k}-1)+g-Lu_k-\lambda e^{V}(e^W-1)-g+LW\\
&=\lambda[e^{v_k}(e^{u_k}-1)-e^{V}(e^W-1)]-L(u_k-W)\\
&\leqslant \lambda[e^{V}(e^{u_k}-1)-e^{V}(e^W-1)]-L(u_k-W)\\
&=\lambda e^V(e^{u_k}-e^W)-L(u_k-W)\\&=(\lambda e^{V+\eta_k}-L)(u_k-W)\\&
\leqslant 0,
\end{aligned}
\end{equation*}
where the function $\eta_k$ satisfies $W\leqslant \eta_k\leqslant u_k\leqslant0$. This implies that $W\leqslant u_{k+1}$ by Lemma~\ref{lm2} and similarly we have $V\leqslant v_{k+1}$. So, we have proved this lemma by the induction.
\end{proof}

Finally, we use these lemmas to prove Theorem~\ref{thm:main1}.

\begin{proof}[Proof of Theorem~\ref{thm:main1}]
For any integers $1\leqslant j\leqslant k$, we have $\Omega_j\subset \Omega_k$. On $\overline{\Omega_j}$, since $u^k\leqslant0$ and $v^k\leqslant0$, one easily sees that $u^k$ and $v^k$ satisfies the conditions in Lemma~\ref{lm3.3}, and we obtain
$$u^k\leqslant u^j, \ \ v^k\leqslant v^j \ \ \text{on}\ \overline{\Omega_j}.$$
Let $\widetilde{u^i}$ and $\widetilde{u^i}$ be the null extension of $u^i$ and $v^i$ on $\mathbb{Z}^n$ respectively.
Note that
$$0\geqslant \widetilde{u^1}\geqslant\widetilde{u^2}\geqslant\ldots\geqslant\widetilde{u^k}\geqslant\ldots,$$
$$0\geqslant \widetilde{v^1}\geqslant\widetilde{v^2}\geqslant\ldots\geqslant\widetilde{v^k}\geqslant\ldots$$
on $\mathbb{Z}^n$.

In order to prove the convergence of the sequences $\{\widetilde{u^i}\}$ and $\{\widetilde{v^i}\}$ to a nontrivial limit, we need to give a uniform lower bound of $\{\widetilde{u^i}\}$ and $\{\widetilde{v^i}\}$. We argue by contradiction. Without loss of generality, we may assume that
$$\lim_{i\rightarrow+\infty}\|\widetilde{u^i}\|_{l^\infty}=+\infty.$$
We define
$$A_1^i=\{x\in\Omega_i:\ -B\leqslant \widetilde{u^i}(x)\leqslant0\},$$
$$A_2^i=\{x\in\Omega_i:\ \widetilde{u^i}(x)<-(2n+1)B-\lambda\},$$
$$A_3^i=\{x\in\Omega_i:\ -(2n+1)B-\lambda\leqslant\widetilde{u^i}(x)<-B\}.$$
By condition assumption, we get $A_2^i\neq\emptyset$ when $i\geqslant i_0$ for some $i_0$. In the following, we only consider $i\geqslant i_0$. Following the proof of Lemma~\ref{lm3.2}, we have
$$\sum_{x\in B(\Omega_i)}|\widetilde{u^i}(x)|< B,$$
which yields $B(\Omega_i)\subset A_1^i$ so that $A_1^i\neq \emptyset$. To obtain the uniform $l^\infty$-norm, we show the contradiction in three steps.

$(i)$. We claim that
$$A_3^i\neq\emptyset.$$
Suppose that $A_3^i=\emptyset$, then $A_1^i\cup A_2^i=\Omega_i$, and there exist two vertices $x,y\in \Omega_i$, satisfying
$x\sim y,\ x\in A_1^i,\ y\in A_2^i$. Note that
\begin{equation*}
\begin{aligned}
&\Delta \widetilde{u^i}(x)=\sum_{z\sim x}(\widetilde{u^i}(z)-\widetilde{u^i}(x))\leqslant \ \widetilde{u^i}(y)-2n\widetilde{u^i}(x)\\
&< -(2n+1)B-\lambda+2nB=-B-\lambda,
\end{aligned}
\end{equation*}
and
\begin{equation*}
\begin{aligned}
&|\Delta \widetilde{u^i}(x)|\leqslant |g(x)|+\lambda|e^{\widetilde{v^i}}(1-e^{\widetilde{u^i}})|<B+\lambda.
\end{aligned}
\end{equation*}
This yields a contradiction. Thus $A_3^i\neq\emptyset$.

$(ii)$. We claim that
\begin{equation}				
\label{A2}				
    \lim_{i\rightarrow+\infty}|A_2^i|=+\infty.
\end{equation}	
There exists a sequence $\{a_m\}_{m=0}^\infty$ such that $$a_0=-(2n+1)B-\lambda,\ \ a_m=2na_{m-1}-(B+\lambda),\ \ m\geqslant1.$$
It's easy to obtain that $a_k<a_{k-1} ,\forall k\geqslant1,$ and $   \lim_{m\rightarrow+\infty}a_m=-\infty.$

Suppose that $min_{x\in \Omega_i}\widetilde{u^i}(x)=\widetilde{u^i}(x_0)\in [a_k,a_{k-1}),$ then there must be a sequence of points $\{z_j\}_{j=1}^k\subset\Omega_i,$ such that $\widetilde{u^i}(z_j)\in [a_j,a_{j-1}).$ Otherwise, there exists $j_0\in\mathbb{N},1\leqslant j_0<k,$ such that $$\widetilde{u^i}(x)\notin [a_{j_0},a_{j_0-1}),\ \ \forall x\in \Omega_i.$$ 
Take $x_1\in B(\Omega_i),$ then $\widetilde{u^i}(x_1)>a_0.$ By the connectivity of $\Omega_i$, there is a path:$$x_0=y_0\sim y_1\sim y_2\sim \dots \sim y_p\sim x_1=y_{p+1},\ \ \{y_t\}_{t=1}^p\subset \Omega_i.$$ 
Then there must be a point $y_r\in \{y_t\}_{t=0}^p,$ satisfying $y_{r+1}\geqslant a_{j_0-1}, y_r<a_{j_0}.$ Note that
\begin{equation*}
\begin{aligned}
&\Delta \widetilde{u^i}(y_{r+1})=\sum_{z\sim y_{r+1}}(\widetilde{u^i}(z)-\widetilde{u^i}(y_{r+1}))\leqslant \ \widetilde{u^i}(y_r)-2n\widetilde{u^i}(y_{r+1})\\
&<a_{j_0}-2na_{j_0-1}\\
&=-B-\lambda,
\end{aligned}
\end{equation*}
and
\begin{equation*}
\begin{aligned}
&|\Delta \widetilde{u^i}(x)|\leqslant |g(x)|+\lambda|e^{\widetilde{v^i}}(1-e^{\widetilde{u^i}})|<B+\lambda.
\end{aligned}
\end{equation*}
This yields a contradiction, thus we have $|A_2^i|\geqslant k.$ Since $$\lim_{i\rightarrow+\infty}\|\widetilde{u^i}\|_{l^\infty}=+\infty,$$ we obtain $\lim _{i\rightarrow\infty}|A_2^i|=+\infty.$

$(iii)$. From $(i),(ii)$, we want to prove that
$$\limsup_{i\rightarrow+\infty}|A_3^i|=+\infty.$$
We argue by contradiction.

Suppose that
$$\limsup_{i\rightarrow+\infty}|A_3^i|= N<+\infty.$$
We focus on the set $\Omega_i\setminus A_3^i=A_1^i\cup A_2^i$, which can be divided into a union of finite disjoint connected subsets, i.e.
$$\Omega_i\setminus A_3^i=\bigcup_{j=1}^{l_i}O_{j,i}.$$
Hence
$$\delta O_{j,i}\subset \delta\Omega_i\cup A_3^i.$$

From the proof of $(i)$,
$$O_{j,i}\subset A_1^i\ \text{or}\ O_{j,i}\subset A_2^i.$$
For some $1\leqslant \widetilde{l}_i\leqslant l_i-1$, without loss of generality, we may assume that
$$A_1^i=\bigcup_{j=1}^{\widetilde{l}_i}O_{j,i},A_2^i= \bigcup_{j=\widetilde{l}_i+1}^{l_i}O_{j,i}.$$
Since $B(\Omega_i)\subset A_1^i$, 
$$\delta\Omega_i\cap \delta O_{j,i}=\emptyset,\ \ \mathrm{for}\ \  \widetilde{l}_i+1\leqslant j\leqslant l_i.$$
Thus for $\widetilde{l}_i+1\leqslant j\leqslant l_i$,
$$\delta O_{j,i}\subset A_3^i.$$
For any $x\in\Omega_i$, since $O_{1,i},\ldots,O_{l_i,i}$ are disjoint connected sets and $|\delta \{x\}|=2n$, there are no more than $2n$ sets from the family $\{O_{j,i}\}_{j=1}^{l_i}$ satisfying $x\in \delta O_{j,i}$.

By the isoperimetric inequality in Lemma~\ref{lm:iso}, we have the following estimate
\begin{equation*}
\begin{aligned}
|A_2^i|&= \sum_{j=\widetilde{l}_i+1}^{l_i}|O_{j,i}|\leqslant \left(\frac{1}{C_n}\right)^{\frac{n}{n-1}}\sum_{j=\widetilde{l}_i+1}^{l_i}|\delta O_{j,i}|^{\frac{n}{n-1}}\\
&\leqslant\left(\frac{1}{C_n}\right)^{\frac{n}{n-1}}\left(\sum_{j=\widetilde{l}_i+1}^{l_i}|\delta O_{j,i}|\right)^{\frac{n}{n-1}}
\leqslant \left(\frac{2n}{C_n}\right)^{\frac{n}{n-1}}|A_3^i|^{\frac{n}{n-1}}.
\end{aligned}
\end{equation*}
So,$$\limsup_{i\rightarrow+\infty}|A_2^i|\leqslant \left(\frac{2n}{C_n}\right)^{\frac{n}{n-1}}   \limsup_{i\rightarrow+\infty}|A_3^i|^{\frac{n}{n-1}}\leqslant \left(\frac{2nN}{C_n}\right)^{\frac{n}{n-1}},$$
which contradicts (\ref{A2}).

From (\ref{e:3.2}), we get
\begin{equation*}
\begin{aligned}
&\sum_{x\in B(\Omega_i)}|\widetilde{u^i}(x)|+\lambda\sum_{x\in\Omega_i}e^{\widetilde{v^i}}(1-e^{\widetilde{u^i}})
=\sum_{x\in\Omega_i}g(x),\\
&\sum_{x\in B(\Omega_i)}|\widetilde{v^i}(x)|+\lambda\sum_{x\in\Omega_i}e^{\widetilde{u^i}}(1-e^{\widetilde{v^i}})
=\sum_{x\in\Omega_i}h(x),
\end{aligned}
\end{equation*}
which implies that
$$\sum_{x\in\Omega_i}e^{\widetilde{u^i}}+e^{\widetilde{v^i}}-2e^{\widetilde{u^i}} \cdot e^{\widetilde{v^i}} \leqslant \frac{B}{\lambda}.$$
We write $$a=e^{-(2n+1)B-\lambda},\ \ b=e^{-B},$$ then 
$$e^{\widetilde{u^i}(x)}\in [a,b]\subset (0,1),\ \ \forall x\in A_3^i.$$
By the mean value inequality, we have the following estimate:
\begin{equation*}
\begin{aligned}
\frac{B}{\lambda} & \geqslant \sum_{x\in A_3^i}e^{\widetilde{u^i}}+e^{\widetilde{v^i}}-2e^{\widetilde{u^i}} \cdot e^{\widetilde{v^i}}\\
&\geqslant \sum_{x\in A_3^i}e^{\widetilde{u^i}}+e^{\widetilde{v^i}}-\frac{1}{2}(e^{\widetilde{u^i}}+e^{\widetilde{v^i}})^2\\
&=\sum_{x\in A_3^i}(e^{\widetilde{u^i}}+e^{\widetilde{v^i}})(1-\frac{1}{2}(e^{\widetilde{u^i}}+e^{\widetilde{v^i}}))\\
&\geqslant \sum_{x\in A_3^i}a(1-\frac{1}{2}(1+b))=\frac{1}{2}a(1-b)|A_3^i|.
\end{aligned}
\end{equation*}
This yields a contradiction to $\limsup_{i\rightarrow+\infty}|A_3^i|=+\infty.$ 

Thus, $\{\widetilde{u^i}\}$ and $\{\widetilde{v^i}\}$ have a uniform bound in $l^\infty(\mathbb{Z}^n)$, and we get the pointwise convergence
\begin{equation*}
\begin{aligned}
&\lim_{i\rightarrow+\infty}\widetilde{u^i}(x)=u^*(x),\ \forall x\in\mathbb{Z}^n,\\
&\lim_{i\rightarrow+\infty}\widetilde{v^i}(x)=v^\ast(x),\ \forall x\in\mathbb{Z}^n,
\end{aligned}
\end{equation*}
where $(u^*,v^*)\in l^\infty(\mathbb{Z}^n)\times l^\infty(\mathbb{Z}^n)$ and satisfies the system $(\ref{uve})$ on $\mathbb{Z}^n$. Furthermore, we pass to the limit, and get that for any $i\geqslant1$,
\begin{equation*}
\begin{aligned}
&e^{\inf_{x\in \mathbb{Z}^n}v^*(x)}\sum_{x\in\Omega_i}(1-e^{u^*})\leqslant \sum_{x\in\Omega_i}e^{v^*}(1-e^{u^*})=\lim_{j\rightarrow \infty}\sum_{x\in\Omega_i}e^{\widetilde{v^j}}(1-e^{\widetilde{u^j}})\leqslant\frac{B}{\lambda},\\
&e^{\inf_{x\in \mathbb{Z}^n}u^*(x)}\sum_{x\in\Omega_i}(1-e^{v^*})\leqslant \sum_{x\in\Omega_i}e^{u^*}(1-e^{v^*})=\lim_{j\rightarrow \infty}\sum_{x\in\Omega_i}e^{\widetilde{u^j}}(1-e^{\widetilde{v^j}})\leqslant\frac{B}{\lambda},
\end{aligned}
\end{equation*}
which yields that $(u^*,v^*)$ is a topological solution.

From Lemma \ref{lm3.2}, we have already seen that $u_\Omega\leqslant0$ and $v_\Omega\leqslant0$. Thus, for any $x\in\mathbb{Z}^n$ we have $u^*(x)\leqslant0,\ \ v^*(x)\leqslant0.$

By Lemma~\ref{lm3.3}, we obtain that the solution $(u_\Omega,v_\Omega)$ is maximal. On $\mathbb{Z}^n$, we suppose that there exists another topological solution $(f_1,f_2)$ of (\ref{uvf}). From the proof of Lemma~\ref{lm3.3}, we observe that $f_1\leqslant0$ and $f_2\leqslant 0$ on $\mathbb{Z}^n$. Applying Lemma~\ref{lm3.3} on $\Omega_i$, we have
$$f_1 \leqslant u^i,\ \ \mathrm{and} \ \ f_2 \leqslant v^i.$$
For a fixed integer $k\geqslant1$, and for $i\geqslant k$ we have
\begin{equation*}
\left\{
\begin{aligned}
&f_1(x) \leqslant \lim_{i\rightarrow\infty}u^i(x)=u^*(x),\ \ \forall x\in\Omega_k,\\
&f_2(x) \leqslant \lim_{i\rightarrow\infty}v^i(x)=v^*(x),\ \ \forall x\in\Omega_k.
\end{aligned}
\right.
\end{equation*}
Hence, for $\forall x\in\mathbb{Z}^n$, there exists a sufficiently large integer $k$ satisfying $x\in \Omega_k$ such that $f_1(x)\leqslant u^*(x)$ and $f_2(x)\leqslant v^*(x)$. Thus we obtain $f_1\leqslant u^*$ and $f_2\leqslant v^*$ on $\mathbb{Z}^n$, and the solution $(u^*,v^*)$ is maximal among all possible solutions.

Finally, for any topological solution $(u,v)$ of $(\ref{uvf})$ and $\epsilon\in(0,1)$, we will prove the decay estimate
$$u= O(e^{-m(1-\epsilon) d(x)}),\ \ v= O(e^{-m(1-\epsilon) d(x)}),$$
where $m=\ln(1+\frac{\lambda}{2n})$.
Note that the solution $(u,v)$ satisfies
\begin{equation*}
\left\{
\begin{aligned}
&\Delta u=\lambda e^v(e^u-1)\ \ \text{on}\ \overline{\Omega_0}^c,\\
&\Delta v=\lambda e^u(e^v-1)\ \ \text{on}\ \overline{\Omega_0}^c.
\end{aligned}
\right.
\end{equation*}
Since
$$\lim_{d(x)\rightarrow+\infty}u(x) =0,\ \ \lim_{d(x)\rightarrow+\infty}v(x) =0,$$
for any $0<\epsilon<1$, we can choose $R\geqslant1$ sufficiently large such that
$$\lambda e^{u(x)+v(x)}\geqslant 2n\left[\left(1+\frac{\lambda}{2n}\right)^{1-\epsilon}-1\right],\ \ d(x)\geqslant R.$$
Then for $d(x)\geqslant R$,
\begin{equation*}
\begin{aligned}
&\Delta u=\lambda e^v(e^u-1)=\lambda e^{v+\xi}u\leqslant \lambda e^{u+v}u \leqslant c_3 u,\\
&\Delta v=\lambda e^u(e^v-1)=\lambda e^{u+\eta}v\leqslant \lambda e^{u+v}v \leqslant c_3 v,
\end{aligned}
\end{equation*}
where $c_3=2n\left[\left(1+\frac{\lambda}{2n}\right)^{1-\epsilon}-1\right],$ the functions $\xi$ and $\eta$ satisfy $u\leqslant\xi\leqslant 0,\ \ v\leqslant \eta \leqslant0.$

Consider the function $h(x)= -e^{-m(1-\epsilon)d(x)}$. For $x\in \overline{\Omega_0}^c$ and $d(x)\geqslant R$, suppose that $d(x)=t\geqslant R\geqslant1$, and we have
$$\Delta h(x)=\sum_{y\sim x}(h(y)-h(x))=\sum_{i=1}^n(h(x+e_i)+h(x-e_i)-2h(x)).$$
If $x_i\neq0$, then
$$h(x+e_i)+h(x-e_i)-2h(x)=-e^{-m(1-\epsilon)(t-1)}-e^{-m(1-\epsilon)(t+1)}+2e^{-m(1-\epsilon)t}.$$
If $x_i=0$, we have
\begin{equation*}
\begin{aligned}
h(x+e_i)+h(x-e_i)-2h(x)&=-2e^{-m(1-\epsilon)(t+1)}+2e^{-m(1-\epsilon)t}\\
&\geqslant -e^{-m(1-\epsilon)(t-1)}-e^{-m(1-\epsilon)(t+1)}+2e^{-m(1-\epsilon)t}.
\end{aligned}
\end{equation*}
Therefore, we have the following inequality
\begin{equation*}
\begin{aligned}
\Delta h(x)&\geqslant n\left[-e^{-m(1-\epsilon)(t-1)}-e^{-m(1-\epsilon)(t+1)}+2e^{-m(1-\epsilon)t}\right]\\
&=n\left[e^{-m(1-\epsilon)}+e^{m(1-\epsilon)}-2\right]h(x)\\
&=n\left[\frac{1}{1+\frac{c_3}{2n}}+\left(1+\frac{c_3}{2n}\right)-2\right]h(x)\\
&\geqslant n\left[2\left(1+\frac{c_3}{2n}\right)-2\right]h(x)=c_3h(x).
\end{aligned}
\end{equation*}
Fix a subset
$$\Omega_0'=\{x\in\mathbb{Z}^n: d(x)\geqslant R_1\geqslant R\},$$
which satisfies $\Omega_0'\cap \overline{\Omega_0}=\emptyset$. By choosing a large constant $C(\epsilon)$, we obtain
\begin{equation*}
\left\{
\begin{aligned}
&(\Delta-c_3)(C(\epsilon)h-u)\geqslant0\ \ \text{on}\ \Omega_0',\\
&(\Delta-c_3)(C(\epsilon)h-v)\geqslant0\ \ \text{on}\ \Omega_0',\\
&\lim_{d(x)\rightarrow+\infty}(C(\epsilon)h-u)(x)=0,\\
&\lim_{d(x)\rightarrow+\infty}(C(\epsilon)h-v)(x)=0,\\
&C(\epsilon)h(x)-u(x)\leqslant 0 \ \ \text{if}\ d(x)=R_1,\\
&C(\epsilon)h(x)-v(x)\leqslant 0 \ \ \text{if}\ d(x)=R_1.
\end{aligned}
\right.
\end{equation*}
These imply that
\begin{equation*}
\left\{
\begin{aligned}
&0\geqslant u(x)\geqslant -C(\epsilon) e^{-m(1-\epsilon)d(x)}\ \ \text{on}\ \Omega_0',\\
&0\geqslant v(x)\geqslant -C(\epsilon) e^{-m(1-\epsilon)d(x)}\ \ \text{on}\ \Omega_0',
\end{aligned}
\right.
\end{equation*}
and complete the proof of Theorem~\ref{thm:main1}. As a consequence, we obtain $(u,v)\in l^p(\mathbb{Z}^n)\times l^p(\mathbb{Z}^n)$, $1\leqslant p\leqslant\infty$.
\end{proof}

\section{Proof of Theorem 1.2}\label{section-Proof of theorem 1.2}

\begin{proof}[Proof of Theorem~\ref{thm1.2}]
For a fixed pair of functions $g$ and $h$, let $(u_\lambda,v_\lambda)$ be a topological solution of $(\ref{uvf})$.

$(a).$By Theorem \ref{thm:main1}, we know 
    $$u_\lambda= O(e^{-m(1-\epsilon) d(x)}),v_\lambda= O(e^{-m(1-\epsilon) d(x)}).$$
So, by Green's identity, we obtain 

\begin{equation*}
\begin{aligned}
&\lambda\sum_{x\in\mathbb{Z}^n}e^{v_\lambda}(1-e^{u_\lambda})
=\sum_{x\in\mathbb{Z}^n}g(x),\\ &\lambda\sum_{x\in\mathbb{Z}^n}e^{u_\lambda}(1-e^{v_\lambda})
=\sum_{x\in\mathbb{Z}^n}h(x).
\end{aligned}
\end{equation*}
This implies that
$$\sum_{x\in\mathbb{Z}^n}\left(e^{u_\lambda}+e^{v_\lambda}-2e^{u_\lambda} \cdot e^{v_\lambda}\right) \leqslant \frac{B}{\lambda}.$$
By the mean value inequality, we have
$$\sum_{x\in\mathbb{Z}^n}\left[e^{u_\lambda}+e^{v_\lambda}-\frac{1}{2}(e^{u_\lambda}+e^{v_\lambda})^2\right]\leqslant \frac{B}{\lambda}.$$
Hence, there holds 
\begin{equation*}
\begin{aligned}
&e^{u_\lambda(x)}+e^{v_\lambda(x)}\leqslant1-\sqrt{1-\frac{2B}{\lambda}} \ \ \text{or} \ \  e^{u_\lambda(x)}+e^{v_\lambda(x)}\geqslant1+\sqrt{1-\frac{2B}{\lambda}}, \ \ \ \ \forall x\in\mathbb{Z}^n 
\\ \Rightarrow &e^{u_\lambda(x)}\leqslant1-\sqrt{1-\frac{2B}{\lambda}} \ \ \text{or} \ \  e^{u_\lambda(x)}\geqslant\sqrt{1-\frac{2B}{\lambda}}, \ \ \ \ \forall x\in\mathbb{Z}^n
\\ \Rightarrow &u_\lambda(x)\leqslant\ln(1-\sqrt{1-\frac{2B}{\lambda}}) \ \ \text{or} \ \  u_\lambda(x)\geqslant\ln(\sqrt{1-\frac{2B}{\lambda}}), \ \ \ \ \forall x\in\mathbb{Z}^n.
\end{aligned}
\end{equation*}

We claim that if $\lambda>2B(2n+e^{4B})$, then 
\begin{equation}
\label{e:ln}
\begin{aligned}
u_\lambda(x)\geqslant\ln(\sqrt{1-\frac{2B}{\lambda}}), \ \ \ \ \forall x\in\mathbb{Z}^n.
\end{aligned}
\end{equation}
We argue by contradiction.

By the connectivity and $\lim_{d(x)\rightarrow+\infty}u(x)=0$, we can find $x_1\sim x_2$ which satisfies $$u_\lambda(x_1)\geqslant\ln(\sqrt{1-\frac{2B}{\lambda}})\ \ \text{and} \ \ u_\lambda(x_2)\leqslant\ln(1-\sqrt{1-\frac{2B}{\lambda}}).$$
 Note that
\begin{equation*}
\begin{aligned}
&\Delta u_\lambda(x_1)\leqslant u_\lambda(x_2)-2nu_\lambda(x_1)
\end{aligned}
\end{equation*}
and
\begin{equation*}
\begin{aligned}
&|\Delta u_\lambda(x)|\leqslant |g(x)|+\lambda e^{v_\lambda(x)}(1-e^{u_\lambda(x)})\leqslant B+\lambda\sum_{x\in\mathbb{Z}^n}e^{v_\lambda}(1-e^{u_\lambda})\leqslant 2B,
\end{aligned}
\end{equation*}
thus we have 
\begin{equation*}
\begin{aligned}
&2nu_\lambda(x_1)\leqslant u_\lambda(x_2)+2B\\
\Rightarrow&2n\ln(\sqrt{1-\frac{2B}{\lambda}})\leqslant \ln(1-\sqrt{1-\frac{2B}{\lambda}})+2B\leqslant \ln(\sqrt{\frac{2B}{\lambda}})+2B\\
\Rightarrow&(1-\frac{2B}{\lambda})^{2n}\leqslant\frac{2B}{\lambda}e^{4B}.
\end{aligned}
\end{equation*}
Since $$(1-\frac{2B}{\lambda})^{2n}\geqslant1-2n\cdot\frac{2B}{\lambda},$$
we obtain 
\begin{equation*}
\begin{aligned}
&1-2n\cdot\frac{2B}{\lambda}\leqslant\frac{2B}{\lambda}e^{4B}\\
\Rightarrow&\lambda\leqslant2B(2n+e^{4B}),
\end{aligned}
\end{equation*}
this yields a contradiction to $\lambda>2B(2n+e^{4B}).$ From the above discussion, we have proved claim $(\ref{e:ln})$.

Similarly we have $$v_\lambda(x)\geqslant\ln(\sqrt{1-\frac{2B}{\lambda}}), \ \ \ \ \forall x\in\mathbb{Z}^n.$$
So, we obtain $$u_\lambda(x)+v_\lambda(x)\geqslant\ln(1-\frac{2B}{\lambda}), \ \ \ \ \forall x\in\mathbb{Z}^n.$$

$(b).$ Let $(u_\lambda^\ast ,v_\lambda^\ast)$ be the maximal topological solution of $(\ref{uvf})$ established in Theorem \ref{thm:main1}. We first prove that the map $\lambda \mapsto (u_\lambda^\ast ,v_\lambda^\ast)$
 is monotone.

For $\lambda_2>\lambda_1>0,$ it is obvious to see that
\begin{equation*}
\left\{
\begin{aligned}
&\Delta u_{\lambda_1}^\ast=\lambda_1e^{v_{\lambda_1}^\ast}(e^{u_{\lambda_1}^\ast}-1)+g\geqslant \lambda_2e^{v_{\lambda_1}^\ast}(e^{u_{\lambda_1}^\ast}-1)+g,\\
&\Delta v_{\lambda_1}^\ast=\lambda_1e^{u_{\lambda_1}^\ast}(e^{v_{\lambda_1}^\ast}-1)+h\geqslant \lambda_2e^{u_{\lambda_1}^\ast}(e^{v_{\lambda_1}^\ast}-1)+h.
\end{aligned}
\right.
\end{equation*}
From Lemma \ref{lm3.3} for $\lambda=\lambda_2$ and the process of taking limits, we conclude that $u_{\lambda_1}^\ast\leqslant u_{\lambda_2}^\ast$ and $v_{\lambda_1}^\ast\leqslant v_{\lambda_2}^\ast$ on $\mathbb{Z}^n$.
Next we shall divide our discussion into two cases:

 Case $1.$ We consider the asymptotic behaviors of topological solutions $(u_{\lambda},v_{\lambda})$ of $(\ref{uve})$ on $\mathbb{Z}^2.$

If $g\equiv0$, then
\begin{equation*}
\left\{
\begin{aligned}
& \Delta (-u_{\lambda})=\lambda e^{v}(1-e^{u})\geqslant0\  \ \ \text{in}\ \mathbb{Z}^2,\\
& \lim_{d(x)\rightarrow+\infty}-u_{\lambda}(x)=0 .\\
\end{aligned}
\right.
\end{equation*}
By Lemma \ref{lm2} and Theorem \ref{thm:main1}, we have $u_{\lambda}\equiv0.$ Thus 
$$\lim_{\lambda\rightarrow0_+}u_\lambda(x)=0,\forall x\in\mathbb{Z}^2.$$
Similarly, if $h\equiv0,$ we have $$\lim_{\lambda\rightarrow0_+}v_\lambda(x)=0,\forall x\in\mathbb{Z}^2.$$

If $g\not\equiv0$. Suppose that there exists $u\in C(\mathbb{Z}^2)$ such that 
$$u_{\lambda}^\ast(x)\rightarrow u(x) \ \ as \ \ \lambda\rightarrow 0_+$$
for all $x\in \mathbb{Z}^2.$ It is obvious that
\begin{equation*}
\left\{
\begin{aligned}
&\Delta u(x)=g(x) \  &\forall x\in \mathbb{Z}^2,\\
&u(x)<0 \ \ \ &\forall x\in \mathbb{Z}^2.
\end{aligned}
\right.
\end{equation*}

For $d\in\mathbb{N}_+$, let
$$Q_d=\{ (x_1,x_2)\in\mathbb{Z}^2:\ |x_1|\leqslant d,\ |x_2|\leqslant d.\}$$
One can find $d_0\in\mathbb{N}_+$ such that $\Omega_0\subset Q_{d_0}.$
For any positive integer $d>d_0$, we set $\omega_d$
to be the unique solution of the following equation
\begin{equation*}
\left\{
\begin{aligned}
\Delta \omega_d=g \ \ \ \ &\mathrm{in}\  Q_d,\\
\omega_d=0 \ \ \ \ \ \  \ &\mathrm{on}\  \delta Q_d.
\end{aligned}
\right.
\end{equation*}
Obviously, $0\geqslant \omega_d\geqslant u,\ \ \mathrm{in}\ Q_d.$

By Lemma~\ref{lm1} and Cauchy inequality, we obtain
\begin{equation*}
\begin{aligned}
&-4\pi\sum_{j=1}^{k_1}m_j\omega_d(p_j)
=-\int_{Q_d}4\pi\sum_{j=1}^{k_1}m_j\delta_{p_j}\omega_d
=-\int_{Q_d}(\Delta\omega_d)\omega_d
=D_{Q_d}(\omega_d)\\
&=\frac{1}{2}\sum_{\substack{x,y\in Q_d\\ x\sim y}}(\omega_d(x)-\omega_d(y))^2+\sum_{\substack{x\in Q_d,y\in \delta Q_d\\ x\sim y}}(\omega_d(x)-\omega_d(y))^2\\
&\geqslant\sum_{i\geqslant d_0}^d \ \ \sum_{\substack{x\in Q_{i-1},y\in \delta Q_i\\ x\sim y}}(\omega_d(x)-\omega_d(y))^2
=\sum_{i\geqslant d_0}^d \sum_{y\in\delta Q_i}(\frac{\partial \omega_d}{\partial \vec{n}}(y))^2\\
&\geqslant\sum_{i\geqslant d_0}^d [\frac{1}{8i+4}(\int_{Q_i}\Delta\omega_d)^2]
=(4\pi\sum_{j=1}^{k_1}m_j)^2\sum_{i\geqslant d_0}^d\frac{1}{8i+4}.
\end{aligned}
\end{equation*}
Since $g \not\equiv 0$, we can get that
$$4\pi\sum_{j=1}^{k_1}m_j\omega_d(p_j)\rightarrow -\infty \ \ \mathrm{as} \ d\rightarrow +\infty,$$
this constradicts $0\geqslant \omega_d\geqslant u.$

So, we can find $x_0\in\mathbb{Z}^2$ such that
$$u_{\lambda}^\ast(x_0)\rightarrow-\infty \ \ as \ \ \lambda\rightarrow 0_+.$$
 If $y\sim x_0$, note that
\begin{equation*}
\begin{aligned}
&\Delta u_{\lambda}^\ast(y)\leqslant u_{\lambda}^\ast(x_0)-2nu_{\lambda}^\ast(y)
\end{aligned}
\end{equation*}
and
\begin{equation*}
\begin{aligned}
&|\Delta u_{\lambda}^\ast(x)|\leqslant |g(x)|+\lambda e^{v_{\lambda}^\ast(x)}(1-e^{u_{\lambda}^\ast(x)})\leqslant B+\lambda\sum_{x\in\mathbb{Z}^n}e^{v_{\lambda}^\ast(x)}(1-e^{u_{\lambda}^\ast(x)})\leqslant 2B,
\end{aligned}
\end{equation*}
then we have 
$$\lim_{\lambda\rightarrow 0_+}u_{\lambda}^\ast(y)\leqslant \lim_{\lambda\rightarrow 0_+}\frac{1}{2n}(u_{\lambda}^\ast(x_0)+2B)=-\infty.$$
By the connectivity, for all $x\in \mathbb{Z}^2,$ we have
$$u_{\lambda}^\ast(x)\rightarrow-\infty \ \ \mathrm{as} \ \lambda\rightarrow 0_+.$$
Similarly, if $h\not\equiv0$, $$v_{\lambda}^\ast(x)\rightarrow-\infty \ \ \mathrm{as} \ \lambda\rightarrow 0_+.$$

Case $2.$ We consider the asymptotic behaviors of topological solutions $(u_{\lambda},v_{\lambda})$ of $(\ref{uve})$ on $\mathbb{Z}^n,\ n\geqslant3.$ 
Let $\psi_n(x)=4\pi\sum\limits_{j=1}^{k_1}m_jG_n(x-p_j), \ x\in \mathbb{Z}^n.$ Applying Lemma \ref{lm:Green function}, we deduce that 
\begin{equation}
\label{e:multgreenfunction}
\left\{
\begin{aligned}
& \Delta \psi_n=4\pi\sum\limits_{j=1}^{k_1}m_j\delta_{p_j}=g\  \ \text{in}\ \mathbb{Z}^n,\\
& \lim_{d(x)\rightarrow+\infty}\psi_n(x)=0 .\\
\end{aligned}
\right.
\end{equation}
Furthermore
\begin{equation*}
\left\{
\begin{aligned}
& \Delta (\psi_n-u_\lambda)=\lambda e^{v_\lambda}(1-e^{u_\lambda})\geqslant0\  \ \ \  \text{in}\ \mathbb{Z}^n,\\
& \lim_{d(x)\rightarrow+\infty}(\psi_n-u_\lambda)(x)=0 .\\
\end{aligned}
\right.
\end{equation*}
Hence, we deduce from Corollary \ref{co1} that 
\begin{equation}
\label{monotonicity}
\begin{aligned}
0\geqslant u_\lambda^\ast(x) \geqslant u_\lambda(x)\geqslant\psi_n(x), \ \ \ \forall x\in\mathbb{Z}^n.
\end{aligned}
\end{equation}
From $(\ref{monotonicity})$ and the monotonicity of $u_{\lambda}^\ast(x)$ with respect to $\lambda$, we can find $u\in C(\mathbb{Z}^n)$ such that  
\begin{equation}
\label{limittou}
\begin{aligned}
\lim_{\lambda\rightarrow0_+}u_{\lambda}^\ast(x)=u(x),\ \ \ \forall x\in\mathbb{Z}^n.
\end{aligned}
\end{equation}
and
\begin{equation}
\label{uistopological}
\begin{aligned}
0\geqslant u(x)\geqslant\psi_n(x), \ \ \ \forall x\in\mathbb{Z}^n.
\end{aligned}
\end{equation}
Using $(\ref{e:multgreenfunction}),(\ref{limittou})$ and $(\ref{uistopological})$, we obtain
\begin{equation*}
\left\{
\begin{aligned}
&\Delta (u-\psi_n)(x)=0, \ \ \ \forall x\in\mathbb{Z}^n,\\
&\lim_{d(x)\rightarrow+\infty}(u-\psi_n)(x)=0.
\end{aligned}
\right.
\end{equation*}
By Corollary \ref{Extre}, we known $u=\psi_n.$ 

Thus, applying $(\ref{monotonicity})$ in the above we obtain that
$$\psi_n(x)\leqslant\lim_{\lambda\rightarrow0_+}u_\lambda(x)\leqslant\lim_{\lambda\rightarrow0_+}u_\lambda^\ast(x)=\psi_n(x), \ \ \forall x\in\mathbb{Z}^n.$$
Then $$\lim_{\lambda\rightarrow0_+}u_\lambda(x)=\psi_n(x), \ \ \forall x\in\mathbb{Z}^n.$$
Similarly, we have 
$$\lim_{\lambda\rightarrow0_+}v_\lambda(x)=4\pi\sum\limits_{j=1}^{k_2}n_jG_n(x-q_j), \ \ \forall x\in\mathbb{Z}^n.$$
\end{proof}

\section{Proof of Theorem 1.3}\label{section-Proof of theorem 1.3} 
\begin{lemma}\label{lm:small enough}
$$||G_n||_{l^\infty(\mathbb{Z}^n)}\rightarrow0\ \ \ \ \mathrm{as}\ n\rightarrow+\infty.$$
\end{lemma}	
\begin{proof}
By the definition of $G_{n}$ and Corollary \ref{co1}, we get that $\lim_{d(x)\rightarrow+\infty}G_n(x)=0$ and $G_n(x) \leqslant0 \ \mathrm{in} \ \mathbb{Z}^n$. Furthermore, the Green’s function $G_n(x)$ can attain its minimum in $\mathbb{Z}^n.$

We claim that $$G_n(0)=\min_{x\in\mathbb{Z}^n}G_n(x).$$
If it is not true, there exists a path $y_0=0\sim y_1\sim y_2\sim\dots\sim y_p, \ p\geqslant1,$ such that 
$$G_n(y_p)=\min_{x\in\mathbb{Z}^n}G_n(x)<G(0),\ \mathrm{and} \ G_n(y_p)<G_n(y_{p-1}).$$
Thus
$$0=\delta_0(y_p)=\Delta G_n(y_p)=\sum_{y\sim y_p}(G_n(y)-G_n(y_p))\geqslant G_n(y_{p-1})-G_n(y_p)>0,$$
which is a contradiction.

From the Lemma \ref{lm:Green function}, we obtain that
$$|G_n(0)|=\frac{1}{(2\pi)^n} \int_{[-\pi,\pi]^n} \frac{1}{2n -2 \sum_{j=1}^n \cos z_j} \mathrm{d} z.$$
Note that the surface area of the $n$-dimensional unit ball is $$\omega_n=\frac{2\pi^{\frac{n}{2}}}{\Gamma(\frac{n}{2})}$$ and $$1-\cos{z_j}=2\sin^2{\frac{z_j}{2}}\geqslant\frac{2}{\pi^2}z_j^2.$$
For any fixed $l>0$, we have
\begin{equation*}
\begin{aligned}
&\frac{1}{(2\pi)^n} \int_{[-\pi,\pi]^n} \frac{1}{2n -2 \sum_{j=1}^n \cos z_j} \mathrm{d} z\\
\leqslant&\frac{1}{(2\pi)^n} \int_{[-\pi,\pi]^n} \frac{1}{\frac{4}{\pi^2} \sum_{j=1}^n z_j^2} \mathrm{d} z\\
\leqslant&\frac{1}{(2\pi)^n} \int_{B(0,l)}\frac{1}{\frac{4}{\pi^2} \sum_{j=1}^n z_j^2} \mathrm{d} z+\frac{1}{(2\pi)^n} \int_{[-\pi,\pi]^n\backslash B(0,l)}\frac{1}{\frac{4}{\pi^2} \sum_{j=1}^n z_j^2} \mathrm{d} z\\
\leqslant&\frac{\pi^2}{4}\frac{1}{(2\pi)^n}\int_0^l \omega_nr^{n-3} \mathrm{d} r+\frac{\pi^2}{4}\frac{1}{(2\pi)^n}\int_{[-\pi,\pi]^n}\frac{1}{l^2}\mathrm{d}z\\
=&\frac{l^{n-2}}{(n-2)2^{n+1}\pi^{\frac{n}{2}-2}\Gamma(\frac{n}{2})}+\frac{\pi^2}{4l^2}.
\end{aligned}
\end{equation*}
Note that $\Gamma(\frac{n}{2})\geqslant [\frac{n}{2}]!$ for $n\geqslant4,$ thus 
$$\limsup_{n\rightarrow+\infty}|G_n(0)|\leqslant \frac{\pi^2}{4l^2}.$$
By taking $l\rightarrow+\infty,$ we finish the proof.
\end{proof}

\begin{proof}[Proof of Theorem~\ref{thm1.3}]
Let $\eta_n(x)=4\pi\sum\limits_{j=1}^{k_2}n_jG_n(x-q_j), \ x\in \mathbb{Z}^n,$ $n\geqslant3.$ Then, from Lemma \ref{lm:small enough} and the proof of (\ref{thm1.2})$(b)$, we obtain that
\begin{equation*}
\left\{
\begin{aligned}
& 0\geqslant u_\lambda^\ast \geqslant u_\lambda\geqslant\psi_n \ \ \ 
 \ \ \mathrm{on} \ \mathbb{Z}^n,\\
 &0\geqslant v_\lambda^\ast \geqslant v_\lambda\geqslant\eta_n \ \ \ 
 \ \ \mathrm{on} \ \mathbb{Z}^n,\\
 &||\psi_n||_{l^\infty(\mathbb{Z}^n)}\rightarrow0\ \ \ \ \mathrm{as}\ n\rightarrow+\infty,\\
 &||\eta_n||_{l^\infty(\mathbb{Z}^n)}\rightarrow0\ \ \ \ \mathrm{as}\ n\rightarrow+\infty.
\end{aligned}
\right.
\end{equation*}
By Theorem (\ref{thm1.2})$(a)$, there exist two constants $N(g,h)$ and $\lambda(g,h)$ related to $g,h$ such that if $n\geqslant N(g,h)$ or $\lambda\geqslant \lambda(g,h)$, then 
\begin{equation*}
\left\{
\begin{aligned}
& 0\geqslant u_\lambda^\ast \geqslant u_\lambda\geqslant\ln{\frac{1}{2}} \ \ \ 
 \ \ \mathrm{on} \ \mathbb{Z}^n,\\
 &0\geqslant v_\lambda^\ast \geqslant v_\lambda\geqslant\ln{\frac{1}{2}} \ \ \ 
 \ \ \mathrm{on} \ \mathbb{Z}^n.
\end{aligned}
\right.
\end{equation*}

Note that
\begin{equation*}
\begin{aligned}
\Delta (u_\lambda^\ast-u_\lambda)&=\lambda(e^{u_\lambda^\ast+v_\lambda^\ast}-e^{u_\lambda+u_\lambda}-e^{v_\lambda^\ast}+e^{v_\lambda})\\
&=\lambda[e^{v_\lambda^\ast}(e^{u_\lambda^\ast}-e^{u_\lambda})+e^{v_\lambda^\ast+u_\lambda}-e^{v_\lambda+u_\lambda}-(e^{v_\lambda^\ast}-e^{v_\lambda})]\\
&=\lambda[e^{v_\lambda^\ast}(e^{u_\lambda^\ast}-e^{u_\lambda})+(e^{v_\lambda^\ast}-e^{v_\lambda})(e^{u_\lambda}-1)],
\end{aligned}
\end{equation*}
similarly, we get
$$\Delta (v_\lambda^\ast-v_\lambda)=\lambda[e^{u_\lambda^\ast}(e^{v_\lambda^\ast}-e^{v_\lambda})+(e^{u_\lambda^\ast}-e^{u_\lambda})(e^{v_\lambda}-1)].$$
Set $\varphi=u_\lambda^\ast+v_\lambda^\ast-u_\lambda-v_\lambda,$ and suppose $n\geqslant N(g,h)$ or $\lambda\geqslant \lambda(g,h)$, then
\begin{equation*}
\left\{
\begin{aligned}
&\Delta\varphi=\lambda[(e^{u_\lambda^\ast}-e^{u_\lambda})(e^{v_\lambda^\ast}+e^{v_\lambda}-1)+(e^{v_\lambda^\ast}-e^{v_\lambda})(e^{u_\lambda^\ast}+e^{u_\lambda}-1)]\geqslant0,\\
&\varphi(x)\rightarrow0 \ \ \ \ \mathrm{as}\ d(x)\rightarrow+\infty,\\
&\varphi\geqslant0.
\end{aligned}
\right.
\end{equation*}
We claim that $\varphi\equiv 0.$ If not, there exists $x_0$ such that $$\varphi(x_0)=\max_{x\in\mathbb{Z}^n}\varphi(x)>0.$$
Then $$0\leqslant\Delta\varphi(x_0)\leqslant \varphi(y)-\varphi(x_0)\leqslant0$$
for any $y\sim x_0,$ which implies that $\varphi(x)=\varphi(x_0),\ \forall x\in\mathbb{Z}^n.$ This contradicts to $\varphi(x)\rightarrow0 \ \mathrm{as}\ d(x)\rightarrow+\infty.$

So, $\varphi\equiv 0,$ that is $u_\lambda^\ast= u_\lambda$ and $v_\lambda^\ast= v_\lambda.$ 
This finishes the proof of Theorem~\ref{thm1.3}.
\end{proof}

\bibliographystyle{alpha}
\bibliography{references} % see references.bib for bibliography management

\newcommand{\etalchar}[1]{$^{#1}$}
\begin{thebibliography}{KLK{\etalchar{+}}93}

\bibitem[Bar17]{isopermetric}
Martin~T. Barlow.
\newblock {\em Random walks and heat kernels on graphs}, volume 438 of {\em London Mathematical Society Lecture Note Series}.
\newblock Cambridge University Press, Cambridge, 2017.

\bibitem[CC16]{asymptoticR}
Zhi-You Chen and Jann-Long Chern.
\newblock Topological multivortex solutions for the {C}hern-{S}imons system with two {H}iggs particles.
\newblock {\em Comm. Partial Differential Equations}, 41(4):705--731, 2016.

\bibitem[CCTL09]{sup-subsolutions-6}
Jann-Long Chern, Zhi-You Chen, Yong-Li Tang, and Chang-Shou Lin.
\newblock Uniqueness and structure of solutions to the {D}irichlet problem for an elliptic system.
\newblock {\em J. Differential Equations}, 246(9):3704--3714, 2009.

\bibitem[CFL02]{Back3}
Hsungrow Chan, Chun-Chieh Fu, and Chang-Shou Lin.
\newblock Non-topological multi-vortex solutions to the self-dual {C}hern-{S}imons-{H}iggs equation.
\newblock {\em Comm. Math. Phys.}, 231(2):189--221, 2002.

\bibitem[Cho09]{sup-subsolutions-8}
Kwangseok Choe.
\newblock Multiple existence results for the self-dual {C}hern-{S}imons-{H}iggs vortex equation.
\newblock {\em Comm. Partial Differential Equations}, 34(10-12):1465--1507, 2009.

\bibitem[CHS22]{generalizedsystem}
Ruixue Chao, Songbo Hou, and Jiamin Sun.
\newblock Existence of solutions to a generalized self-dual chern-simons system on finite graphs.
\newblock In {\em arXiv preprint arXiv:2206.12863}, 2022.

\bibitem[CI00]{Back2}
Dongho Chae and Oleg~Yu. Imanuvilov.
\newblock The existence of non-topological multivortex solutions in the relativistic self-dual {C}hern-{S}imons theory.
\newblock {\em Comm. Math. Phys.}, 215(1):119--142, 2000.

\bibitem[CY95a]{introchern-simons3}
Luis~A. Caffarelli and Yi~Song Yang.
\newblock Vortex condensation in the {C}hern-{S}imons {H}iggs model: an existence theorem.
\newblock {\em Comm. Math. Phys.}, 168(2):321--336, 1995.

\bibitem[CY95b]{Back1}
Luis~A. Caffarelli and Yi~Song Yang.
\newblock Vortex condensation in the {C}hern-{S}imons {H}iggs model: an existence theorem.
\newblock {\em Comm. Math. Phys.}, 168(2):321--336, 1995.

\bibitem[DJLW98]{introchern-simons4}
Weiyue Ding, J\"urgen Jost, Jiayu Li, and Guofang Wang.
\newblock An analysis of the two-vortex case in the {C}hern-{S}imons {H}iggs model.
\newblock {\em Calc. Var. Partial Differential Equations}, 7(1):87--97, 1998.

\bibitem[Dzi94]{Chern-Simonstheory1}
Jacek Dziarmaga.
\newblock {Low-energy dynamics of U(1)**N Chern-Simons solitons}.
\newblock {\em Phys. Rev. D}, 49:5469--5479, 1994.

\bibitem[Gri18]{bookanalysisgraphs}
Alexander Grigor'yan.
\newblock {\em Introduction to analysis on graphs}, volume~71 of {\em University Lecture Series}.
\newblock American Mathematical Society, Providence, RI, 2018.

\bibitem[HHW23]{Hua2023TheEO}
Bobo Hua, Genggeng Huang, and Jiaxuan Wang.
\newblock The existence of topological solutions to the chern-simons model on lattice graphs.
\newblock In {\em arXiv preprint arXiv:2310.13905}, 2023.

\bibitem[HK25]{generalizedsystem2}
Songbo Hou and Xiaoqing Kong.
\newblock Existence and asymptotic behaviors of solutions to {C}hern-{S}imons systems and equations on finite graphs.
\newblock {\em Calc. Var. Partial Differential Equations}, 64(3):Paper No. 77, 15, 2025.

\bibitem[HKP90]{introchern-simons1}
Jooyoo Hong, Yoonbai Kim, and Pong~Youl Pac.
\newblock Multivortex solutions of the abelian {C}hern-{S}imons-{H}iggs theory.
\newblock {\em Phys. Rev. Lett.}, 64(19):2230--2233, 1990.

\bibitem[HS11]{sup-subsolutions}
Jongmin Han and Kyungwoo Song.
\newblock The existence and asymptotics of solutions for the abelian {C}hern-{S}imons system with two {H}iggs fields and two gauge fields.
\newblock {\em Nonlinear Anal.}, 74(18):7426--7436, 2011.

\bibitem[HWY21]{JFA}
Hsin-Yuan Huang, Jun Wang, and Wen Yang.
\newblock Mean field equation and relativistic {A}belian {C}hern-{S}imons model on finite graphs.
\newblock {\em J. Funct. Anal.}, 281(10):Paper No. 109218, 36, 2021.

\bibitem[JW90]{introchern-simons2}
R.~Jackiw and Erick~J. Weinberg.
\newblock Self-dual {C}hern-{S}imons vortices.
\newblock {\em Phys. Rev. Lett.}, 64(19):2234--2237, 1990.

\bibitem[KLK{\etalchar{+}}93]{Chern-Simonstheory2}
Chanju Kim, Choonkyu Lee, Pyungwon Ko, Bum-Hoon Lee, and Hyunsoo Min.
\newblock Schr\"odinger fields on the plane with {$[{\rm U}(1)]^N$} {C}hern-{S}imons interactions and generalized self-dual solitons.
\newblock {\em Phys. Rev. D (3)}, 48(4):1821--1840, 1993.

\bibitem[LL10]{bookgreenfunction}
Gregory~F. Lawler and Vlada Limic.
\newblock {\em Random walk: a modern introduction}, volume 123 of {\em Cambridge Studies in Advanced Mathematics}.
\newblock Cambridge University Press, Cambridge, 2010.

\bibitem[LP09]{sup-subsolutions-7}
Chang-Shou Lin and Jyotshana~V. Prajapat.
\newblock Vortex condensates for relativistic abelian {C}hern-{S}imons model with two {H}iggs scalar fields and two gauge fields on a torus.
\newblock {\em Comm. Math. Phys.}, 288(1):311--347, 2009.

\bibitem[LPY07]{Chern-Simonstheory}
Chang-Shou Lin, Augusto~C. Ponce, and Yisong Yang.
\newblock A system of elliptic equations arising in {C}hern-{S}imons field theory.
\newblock {\em J. Funct. Anal.}, 247(2):289--350, 2007.

\bibitem[LSY24]{Topologicaldegree}
Jiayu Li, Linlin Sun, and Yunyan Yang.
\newblock Topological degree for {C}hern-{S}imons {H}iggs models on finite graphs.
\newblock {\em Calc. Var. Partial Differential Equations}, 63(4):Paper No. 81, 21, 2024.

\bibitem[MS22]{Greenfunctionend}
Emmanuel Michta and Gordon Slade.
\newblock Asymptotic behaviour of the lattice {G}reen function.
\newblock {\em ALEA Lat. Am. J. Probab. Math. Stat.}, 19(1):957--981, 2022.

\bibitem[SY95]{Back5}
Joel Spruck and Yi~Song Yang.
\newblock Topological solutions in the self-dual {C}hern-{S}imons theory: existence and approximation.
\newblock {\em Ann. Inst. H. Poincar\'e{} C Anal. Non Lin\'eaire}, 12(1):75--97, 1995.

\bibitem[Tar96a]{Back4}
Gabriella Tarantello.
\newblock Multiple condensate solutions for the {C}hern-{S}imons-{H}iggs theory.
\newblock {\em J. Math. Phys.}, 37(8):3769--3796, 1996.

\bibitem[Tar96b]{sup-subsolutions-9}
Gabriella Tarantello.
\newblock Multiple condensate solutions for the {C}hern-{S}imons-{H}iggs theory.
\newblock {\em J. Math. Phys.}, 37(8):3769--3796, 1996.

\bibitem[WWZ25]{topologicaldegreget3solutions}
Chunhua Wang, Wenju Wu, and Fulin Zhong.
\newblock Brouwer degree for chern-simons higgs models on finite graphs.
\newblock In {\em arXiv preprint arXiv:2504.14510}, 2025.

\end{thebibliography}

\end{document}